\documentclass[11pt]{amsart}    %amsart
 \pdfoutput=1
\usepackage[utf8]{inputenc} 

\usepackage{amsfonts}
\usepackage{amssymb}
\usepackage{amsmath}
\usepackage{amsthm}
\usepackage{mathrsfs}
\usepackage{tikz,enumerate,caption,float}

\newtheorem{theorem}{Theorem}
\newtheorem{lemma}{Lemma}
\newtheorem{corollary}{Corollary}

\newcommand{\Rinf}{\mathbb{R}_{\infty}}
\newcommand{\be}{\begin{equation}}
\newcommand{\ee}{\end{equation}}

\newcommand{\A}{\mathcal{A}}
\newcommand{\D}{\mathcal{D}}
\newcommand{\F}{\mathcal{F}}
\newcommand{\G}{\mathcal{G}}
\newcommand{\U}{\mathcal{U}}
\newcommand{\X}{\mathcal{X}}
\newcommand{\R}{\mathcal{R}}
\newcommand{\Y}{\mathcal{Y}}
\newcommand{\one}{\mathbf{1}}

\title{The Algebras of the GIry monad}
\author{Kirk Sturtz}
\email{kirksturtz@yandex.com}

\begin{document}
\maketitle

\begin{abstract}
The algebras of the Giry monad on the category of measurable spaces are derived.  These algebras, which we denote as $\mathbf{Cvx}_{Meas}$,  are measurable convex spaces satisfying three properties: (1) all the affine sum operations $\prod_{i=1}^n X \xrightarrow{\sum_{i=1}^n p_i \pi_i} X$, where the functions $\pi_i$ are the coordinate projection function, are measurable functions,  (2) $X$ is coseparated by affine measurable maps to $\Rinf$, and (3) $X$ satisfies a fullness property.  The category of algebras has as morphisms affine measurable functions.
\end{abstract}

\vspace{.2in}

 Let $(\G, \eta, \mu)$ denote the Giry monad where $X \xrightarrow{\eta_X} \G(X)$ maps $x \mapsto \delta_x$ where $\delta_x$ is the Dirac measure at $x$, and $\mu_X(Q)[U] = \int_{P \in \G(X)} ev_U(P) 
\, dQ$ for all measurable sets $U$ in $X$ and all $Q \in \G^2(X)$.

Let $\mathbf{Alg}_{\G}$ denote the category of algebras of the $\G$-monad which has as objects those measurable spaces $X$ for which there exists a $\G$-algebra $\G(X) \xrightarrow{h} X$ which is an object in the Eilenberg-Moore category of the $\G$-monad. The morphisms of $\mathbf{Alg}_{\G}$ are those measurable functions $X \xrightarrow{m} Y$ such that $m$ constitutes an arrow in the Eilenberg-Moore category of the $\G$-monad. 
   
If $X$ is any measurable space then the space of probability measures $\G{X}$ has a convex space structure defined pointwise: if $\{P_i\}_{i=1}^{n}$ is finite collection of probability measures on $X$  then, for every sequence $\{p_i\}_{i=1}^{n}$ with each $p_i \in [0,1]$ such that $\sum_{i=1}^{n} p_i = 1$, the affine sum  $\sum_{i=1}^{n} p_i P_i$,  is also a probability measure, defined at the measurable set $U$ in $X$ by
\begin{equation}
\big(\sum_{i=1}^{n} p_i P_i\big)(U) = \sum_{i=1}^{n} p_i P_i(U).
\end{equation}

\begin{lemma} \label{FunThm} Given any $\G$-algebra $\G(X) \xrightarrow{h} X$ the base space $X$ has the structure of a convex space which makes the measurable function $h$ an affine (measurable) map. Moreover, morphisms of $\G$-algebras are also affine maps.
\end{lemma}
\begin{proof} Given $h$ define the convex space structure on $X$ by
\begin{equation}
 \sum_{i=1}^{n} p_i x_i := h(\sum_{i=1}^{n} p_i \delta_{x_i}).
\end{equation}

Because a $\G$-algebra $h$ must satisfy $h \circ \mu_X = h \circ \G{h}$ we have, for any finite sequence $Q_i \in \G{X}$, 
\begin{equation}
\begin{array}{lcl}
(h \circ \mu_X)( \sum_{i=1}^{n} p_i \delta_{Q_i}) &=& (h \circ \G{h})( \sum_{i=1}^{n} p_i \delta_{Q_i}) \\
h( \sum_{i=1}^{n} p_i Q_i) &=& h(\sum_{i=1}^{n} p_i \, \delta_{h(Q_i)}) \\
 &=& \sum_{i=1}^{n} p_i h(Q_i)
\end{array}
\end{equation}
where the last line makes use of the definition of the convex structure on $X$. Thus, every $\G$-algebra is affine.

To prove that any map of $\G$-algebras $f: (X, h) \to (Y, k)$ is an affine map, we compute
\begin{equation}
\begin{array}{lcll}
f(\sum_{i=1}^{n} p_i x_i) &=& f(h(\sum_{i=1}^{n} p_i \delta_{x_i})) &  \\
&=& k(G(f)(\sum_{i=1}^{n} p_i \delta_{x_i})) &  \\
&=& k(\sum_{i=1}^{n} p_i \delta_{f(x_i)}) &  \\
&=& \sum_{i=1}^{n} p_i f(x_i) & 
\end{array}.
\end{equation}

\end{proof}

Lemma 1 shows that any object in $\mathbf{Alg}_{\G}$ lies in both the categories $\mathbf{Meas}$ and $\mathbf{Cvx}$, where $\mathbf{Cvx}$ is the category of convex spaces.  Similarly, any morphism in $\mathbf{Alg}_{\G}$ is also a morphism in both  $\mathbf{Meas}$ and $\mathbf{Cvx}$.  An example of such an object is $\mathbb{R}_{\infty}$ which, as a measurable space, is  the one-point compactification of the real-line with the Borel $\sigma$-algebra. As a convex space, $\mathbb{R}_{\infty}$ has the natural convex space structures on the real-line extended by the point $\infty$ with the property that $p r + (1-p) \infty = \infty$ for all $r \in \mathbb{R}$ and all $p \in [0,1)$.

 Let $\mathbf{Meas} \cap \mathbf{Cvx}$ denote the category whose objects $X$ are convex spaces and, in addition, $X$ also posseses a measurable space structure such that all the algebraic operations of taking  affine sums of elements, \mbox{$\prod_{i=1}^n X \xrightarrow{\sum_{i=1}^n p_i \pi_i} X$}, where $\pi_i$ is the coordinate projection function, is a measurable function.  Note these maps  are always an affine function because of the convex space structure on product spaces.   Moreover, we require that each object in $\mathbf{Meas} \cap \mathbf{Cvx}$ must have enough affine measurable functions $X \rightarrow \mathbb{R}_{\infty}$ to coseparate the points of $X$. The  objects of $\mathbf{Meas} \cap \mathbf{Cvx}$ are called \textit{measurable convex spaces}.  The  morphisms of $\mathbf{Meas} \cap \mathbf{Cvx}$ are affine measurable functions. Because $\mathbb{R}_{\infty}$ is a coseparator in $\mathbf{Cvx}$ it follows that $\mathbb{R}_{\infty}$ is a coseparator in $\mathbf{Meas} \cap \mathbf{Cvx}$. 

The space $\G(X)$ with the pointwise convex space structure and measurable structure is a measurable convex space because
 the evaluation maps $\G(X) \xrightarrow{ev_U} \mathbb{R}_{\infty}$ coseparate any  two distinct probability measures in $\G(X)$ and the operation of taking affine sums is a measurable function because

\begin{lemma} Let $p \in [0,1]$.  The operation of taking an affine sum
\be
\G(X) \times \G(X) \xrightarrow{ p \pi_1 + (1-p) \pi_2} \G(X)
\ee
is a measurable function.  This implies that all finite affine sum operators are measurable.
\end{lemma}
\begin{proof} 
Let $U$ be a measurable set in $X$. 
By the pointwise convex space structure of $\G(X)$ we have the commutative $\mathbf{Cvx}$-diagram
\begin{center}
\begin{tikzpicture}[scale=2]
   \node  (GXGX) at (0,0)  {$\G(X) \times \G(X)$};
   \node (GX)      at  (3,0)  {$\G(X)$};
   \node  (II)      at   (0, -1) {$[0,1] \times [0,1]$};
   \node  (I)      at    (3, -1) {$[0,1]$};

   \draw[->,above,dashed] (GXGX) to node {$p \pi_1 + (1-p) \pi_2$} (GX);
   \draw[->, right] (GX) to node {$ev_{\small{U}}$} (I);
   \draw[->, left] (GXGX) to node {$ev_{\small{U}} \times ev_{\small{U}}$} (II);
   \draw[->,below] (II) to node {$p \pi_1^{'} + (1-p) \pi_2^{'}$} (I);
\end{tikzpicture}
\end{center}
where all the affine maps, with the possible exception of the top affine sum operation, are measurable. 
 Since the south-east path of this diagram is measurable  it follows, for every measurable set $W$ of $[0,1]$, that
\be
(p \pi_1 + (1-p) \pi_2)^{-1}\big(ev_U^{-1}(W)\big) \in \Sigma_{\G(X) \times \G(X)}.
\ee
This last equation is true for all measurable sets $U$ in $X$, and since the affine measurable functions $ev_U$ generate the initial $\sigma$-algebra on $\G(X)$ it follows that the function $(p \pi_1 + (1-p) \pi_2)$ is measurable.

The second statement follows from the observation that we can use induction on the number of components in a product space, and, assuming $n$ components, that we are free to choose $n-1$ parameters $p_i \in [0,1]$ freely.  Since every affine sum is uniquely defined by $n-1$ parameters the result follows.
\end{proof}

Given any measurable space $X$ and any $P \in \G(X)$ the expectation map
\be \nonumber
\begin{array}{ccc}
hom_{\mathbf{Meas}}(X, \mathbb{R}_{\infty}) &\xrightarrow{\hat{\mathbb{E}}_{P}}&  \Rinf \\
f & \mapsto &  \int_X f \, dP
\end{array}.
\ee
 is a functional is   (1) weakly averaging:  $\mathbb{E}_{P}(\overline{c}) = c$ for every constant function $X \xrightarrow{\overline{c}} \Rinf$ with value $c \in\Rinf$, and (2) scalar covariant: $\mathbb{E}_P(\lambda \cdot f) = \lambda \cdot  \mathbb{E}_P(f)$ for all $\lambda \in \Rinf$. 
 If we let $\mathbb{R}_{\infty}^{hom_{\mathbf{Meas}}(X,\mathbb{R}_{\infty})}|_{wa+sc}$ denote the set of all weakly averaging and scalar covariant functionals from $hom_{\mathbf{Meas}}(X, \mathbb{R}_{\infty})$  to $\mathbb{R}_{\infty}$ then we have a bijective    correspondence between  $\G(X)$ and $\mathbb{R}_{\infty}^{hom_{\mathbf{Meas}}(X,\mathbb{R}_{\infty})}|_{wa+sc}$. The correspondence is $P \mapsto \hat{\mathbb{E}}_{P}$ and $J \mapsto \hat{J}$ where the probability measure $\hat{J}$ is defined on a measurable set $U$ by $\hat{J}(U) = J(\chi_U)$.                                                          

Let $\mathbb{R}_{\infty}^X = hom_{\mathbf{Meas} \cap \mathbf{Cvx}}(X, \mathbb{R}_{\infty})$. Taking $X=\mathbb{R}_{\infty}$ we obtain the space $\mathbb{R}_{\infty}^{\mathbb{R}_{\infty}}$ of affine measurable endomaps on $\mathbb{R}_{\infty}$.

 An $\mathbb{R}_{\infty}$-generalized point of an object $X$ in $\mathbf{Meas} \cap \mathbf{Cvx}$ is a functional $\mathbb{R}_{\infty}^X \xrightarrow{J} \mathbb{R}_{\infty}$ satisfying, for all $\phi \in \mathbb{R}_{\infty}^{\mathbb{R}_{\infty}}$ and all $f \in \mathbb{R}_{\infty}^X$, the equation 
\begin{equation}
\phi \big( J(f) \big) = J(\phi \circ f)
\end{equation}
which implies that $J$ is (1) weakly averaging, and (2) scalar covariant.
(Generalized points are defined in Definition 8.19 of \textit{Sets for Mathematics}\cite{SFM}, and several basic properties are discussed therein.)

Note that if $P \in \G(X)$ then the functional $\mathbb{R}_{\infty}^X \xrightarrow{\mathbb{E}_P} \mathbb{R}_{\infty}$, which is the restriction of the functional $\hat{ \mathbb{E}}_P$ to affine measurable functions, is an $\Rinf$-generalized point of $X$ since, for all $f \in \mathbb{R}_{\infty}^X$ and all $\phi \in \mathbb{R}_{\infty}^{ \mathbb{R}_{\infty}}$, 
$\mathbb{E}_P(\phi \circ f) = \phi( \mathbb{E}_P(f) )$.

Since we are restricting the operators $\hat{\mathbb{E}}_{P}$ to operate only on affine measurable functions, yielding the operators $\mathbb{E}_P$, we have
\begin{lemma}
  If $J$ is an $\mathbb{R}_{\infty}$-generalized element of $X$ then there exists a $P \in \G(X)$ such that $J=\mathbb{E}_P$.
\end{lemma}
\begin{proof}
Let $X$ be an object in $\mathbf{Meas} \cap \mathbf{Cvx}$.  We have the inclusion function $hom_{\mathbf{Meas} \cap \mathbf{Cvx}}(X, \mathbb{R}_{\infty}) \xrightarrow{\iota} hom_{\mathbf{Meas}}(X, \mathbb{R}_{\infty})$ which induces the restriction mapping
\be
\mathbb{R}_{\infty}^{hom_{\mathbf{Meas}}(X,\mathbb{R}_{\infty})}|_{wa+sc} \xrightarrow{\mathbb{R}_{\infty}^{\iota}} \mathbb{R}_{\infty}^{hom_{\mathbf{Meas} \cap \mathbf{Cvx}}(X, \mathbb{R}_{\infty})}|_{\Rinf -generalized pt}
\ee
which is a surjective function.  This specifies an equivalence relation on the set $\mathbb{R}_{\infty}^{hom_{\mathbf{Meas}}(X,\mathbb{R}_{\infty})}|_{wa+linear}$ defined by $\hat{\mathbb{E}}_P \cong \hat{\mathbb{E}}_Q$ if and only if the restriction of those functionals are equal on the set of all affine measurable functions $X \rightarrow \mathbb{R}_{\infty}$.  Thus every $\mathbb{R}_{\infty}$-generalized point of $X$ comes from a functional $\hat{\mathbb{E}}_P$, which in turn arises from the probability measure $P$ on $X$.
\end{proof}

We say an object $X$ in $\mathbf{Meas} \cap \mathbf{Cvx}$ satisfies the \textbf{fullness property} if and only if for every $P \in \G(X)$ the property
\begin{equation}
\displaystyle{ \bigcap_{f \in \mathbb{R}_{\infty}^X}} f^{-1}\big(\mathbb{E}_{P}(f)\big) = \{x \in X \, | \, \mathbb{E}_P(f)=f(x) \quad \forall f \in \mathbb{R}_{\infty}^X\} \ne \emptyset
\end{equation}
holds.

Define $\mathbf{Cvx}_{Meas}$ to be the full subcategory of $\mathbf{Meas} \cap \mathbf{Cvx}$ consisting of those objects which satisfy the fullness property.  

The space $\mathbb{R}_{\infty}$ is an object in $\mathbf{Cvx}_{Meas}$ because, for every $P \in \G(\mathbb{R}_{\infty})$, we have $\mathbb{E}_P(id_{\mathbb{R}_{\infty}}) \in \cap_{f \in \mathbb{R}_{\infty}^{\mathbb{R}_{\infty} }} f^{-1}\big(\mathbb{E}_{P}(f)\big)$.  (See exercise 8.23 in \textit{Sets for Mathematics}.)  The space $\mathbf{1}$ is also an object in $\mathbf{Cvx}_{Meas}$.

\begin{theorem} 
The restricted Yoneda embedding functor defined (on objects) by
\begin{equation}
\begin{array}{ccc}
\mathbf{Cvx}_{Meas}^{op} & \xrightarrow{\mathcal{Y}} & \mathbf{Set}^{\mathcal{R}} \\
X & \mapsto & hom_{\mathbf{Cvx}_{Meas}}(X, \bullet)
\end{array}
\end{equation}
is a full and faithful functor. 
\end{theorem}
\begin{proof} 
In the category $\mathbf{Cvx}_{Meas}$ every affine measurable function $X \xrightarrow{g} Y$ is determined by its value on points $\mathbf{1} \xrightarrow{x} X$. Hence to prove the fully faithful property it suffices to prove those properties on points.

Faithful:  Note $\mathcal{Y}(x)$ is the evaluation map
$\mathbb{R}_{\infty}^X \xrightarrow{ev_x} \mathbb{R}_{\infty}$.  Let $\mathbf{1} \xrightarrow{x_i} X$, for $i=1,2$ be two points of $X$. If $f(x_1) = \mathcal{Y}(x_1)f = \mathcal{Y}(x_2)f = f(x_2)$ for all $f \in \mathbb{R}_{\infty}^X$, then since $X$ has enough affine measurable maps to $\mathbb{R}_{\infty}$ to coseparate points it follows that $x_1=x_2$ and $\mathcal{Y}$ is faithful functor.  

Full:  If $J \in Nat( hom(X, \cdot), hom(\mathbf{1}, \cdot) )$  then, at the single component of $\mathcal{R}$, $J$ is a  functional  $\mathbb{R}_{\infty}^X \rightarrow \mathbb{R}_{\infty}$ which is weakly averaging and scalar covariant. In other words,  $J$ is an $\mathbb{R}_{\infty}$-generalized point of $X$.   Now to complete the proof we employ Lemma 3.3:  $J = \mathbb{E}_P$ for some  $P \in \G(X)$.  Then, because $X \in_{ob} \mathbf{Cvx}_{Meas}$, it satisfies the fullness property and we have, for all $f \in \mathbb{R}_{\infty}^X$, the property that $J(f) = \mathbb{E}_P(f) \in Im(f)$.  So there exists an element $x_f \in X_f = \{x \in X \, | \, J(f) = f(x) \}$ such that $J(f)=f(x_f)$. But the fullness property says $\cap_f X_f \ne \emptyset$.  Thus there exist an $x \in X$ such that $J(f) = f(x)$ for all $f \in \mathbb{R}_{\infty}^X$ from which it follows that $\mathcal{Y}(x) = ev_x$ and we conclude that $\mathcal{Y}$ is a full functor.
\end{proof}

\begin{corollary} Let $\mathcal{R}$ denote the full subcategory of $\mathbf{Cvx}_{Meas}$ consisting of the single object $\mathbb{R}_{\infty}$, and let $\iota: \R \hookrightarrow \mathbf{Cvx}_{Meas}$ denote the inclusion functor.  The right Kan extension of $\iota$  along $\iota$ is  the identity functor on $\mathbf{Cvx}_{Meas}$ along with the identity natural transformation $id: id_{\mathbf{Cvx}_{Meas}} \circ \iota \Rightarrow \iota$.  In other words, $\iota$ is a codense functor.
\end{corollary}
\begin{proof}
We can compute the right Kan extension of $\iota$ along $\iota$ pointwise, for each $X \in_{ob} \mathbf{Cvx}_{Meas}$, as $\lim \mathcal{D}_X$ where
$$
\D_X = X \downarrow \iota \xrightarrow{\pi} \mathcal{R} \xrightarrow{\iota} \mathbf{Cvx}_{Meas}
$$
where  $X\downarrow \iota$ denotes the slice category whose objects consist of affine measurable functions $X \xrightarrow{f} \mathbb{R}_{\infty}$, and  $X \downarrow \iota \xrightarrow{\pi} \mathcal{R}$ denotes the projection functor.

We claim that $\lim \D_X = (X, \{f\}_{f \in \mathbb{R}_{\infty}^X})$.   To obtain a contradiction, suppose to the contrary that $X$ is not the limit of $\D_X$.  Then there exists a cone $\one \xrightarrow{\beta} \D_X$  such that there exists no $x \in X$ such that 
\be \label{noX}
\{\beta_f\}_{f \in \mathbb{R}_{\infty}^X}  = \{f(x)\}_{f \in \mathbb{R}_{\infty}^X} = \{ ev_x(f) \}_{f \in \mathbb{R}_{\infty}^X}.
\ee
But  the functional 
\be \nonumber
\begin{array}{ccc}
\Rinf^X & \xrightarrow{\beta_{\bullet}} & \Rinf \\
f & \mapsto & \beta_f
\end{array}
\ee
 is an $\Rinf$-generalized point of $X$  because $\beta$ is a cone over the diagram $\D_X$.   That it to say it satisfies $\beta_{\phi \circ f} = \phi \circ \beta_f$ for all $\phi \in \Rinf^{\Rinf}$ and all $f \in \Rinf^X$.   
But the hypothesis  that there exists no $x \in X$ such that equation (\ref{noX}) holds leads to the contradiction that there does exists an $\Rinf$-generalized point of $X$ given by the functional $\beta_{\bullet}$ which is not an evaluation map $ev_x$ for some $x \in X$.  This contradicts the result in Theorem 1  that the restricted Yoneda embedding $\Y$ is a full functor.   Hence we conclude $(X, \{f\}_{f \in \mathbb{R}_{\infty}^X})= \lim \D_X$.  This proves that the right Kan extension of $\iota$ along $\iota$ consists of  the identity functor on $\mathbf{Cvx}_{Meas}$  from which it readily follows that  the natural transformation $id_{\mathbf{Cvx}_{Meas}} \circ \iota \rightarrow \iota$ is the identity natural transformation.

\end{proof}

%The following result is, using Propositions 1 and 2, page 242, along with Corollary 3 on page 243 of \cite{CWM}, a consequence of Theorem 1. 

\begin{corollary} If $X$ is an object in $\mathbf{Cvx}_{Meas}$ then there exists a unique affine measurable function $\G(X) \xrightarrow{\epsilon_X} X$ such that $\epsilon_X(\delta_x)=x$ for all $x \in X$. 
\end{corollary}
\begin{proof}  From the proof of the preceding corollary 
$X = \lim \mathcal{D}_X$ with the projection map at component $f$ being $f$. 

We can construct a cone over $\mathcal{D}_X$ with vertex $\G(X)$ and natural transformation components $\mathbb{E}_{\bullet}(f) = \mathbb{E}_{\bullet}(id_{\mathbb{R}_{\infty}}) \circ \G(f)$.
\begin{center}
\begin{tikzpicture}{scale=5}
   \node  (X) at  (0,0)   {$X$};
   \node  (R1) at  (1.5,1.5)  {$\Rinf$};
   \node  (R2) at  (1.5, -1.5)  {$\Rinf$};
   \node  (GR1) at  (-1.5, 1.5)  {$\G(\Rinf)$};
   \node  (GR2)  at  (-1.5, -1.5)  {$\G(\Rinf)$};
   \node  (GX)   at   (-2.5, 0)    {$\G(X)$};

   \draw[->, above] (X) to node {$f$} (R1);
   \draw[->, left]  (X) to node {$g$} (R2);
   \draw[->,right] (R1) to node {$\phi$} (R2);
   \draw[->,above] (GR1) to node {$\mathbb{E}_{\bullet}(id_{\Rinf})$} (R1);
   \draw[->, below] (GR2) to node {$\mathbb{E}_{\bullet}(id_{\Rinf})$} (R2);
  \draw[->, left] (GX) to node {$\G(f)$} (GR1);
   \draw[->,left] (GX) to node {$\G(g)$} (GR2);
  \draw[->,above,dashed] (GX) to node {$\epsilon_X$} (X);
\end{tikzpicture}
\end{center}

Since $X=\lim \mathcal{D}_X$ there exists a unique $\mathbf{Cvx}_{Meas}$-morphism   $\G(X) \xrightarrow{\epsilon_X} X$ such that $f \circ \epsilon_X = \mathbb{E}_{\bullet}(f)$ for all affine maps $X \xrightarrow{f} \mathbb{R}_{\infty}$. It follows that for each Dirac measure $\delta_x \in \G(X)$ that, for all $X \xrightarrow{f} \mathbb{R}_{\infty}$ in $\mathbf{Cvx}_{Meas}$ that $f(\epsilon_X(\delta_x)) = f(x)$.  Since $\mathbb{R}_{\infty}$ is a coseparator in $\mathbf{Cvx}_{Meas}$ it follows $\epsilon_X(\delta_x)=x$.
\end{proof}

The defining property of $\epsilon_X$ is that it is the unique affine measurable function such that, for all $f \in \mathbb{R}_{\infty}^X$, the property
$$ 
f \circ \epsilon_X = \mathbb{E}_{\bullet}(id_X) \circ \G(f) = \mathbb{E}_{\bullet}(f)
$$
holds. In the special case of $X = \mathbb{R}_{\infty}$ or  for $X$ a closed convex subset of $\mathbb{R}^n$  it follows that for each fixed $P \in \G(X)$ that $\epsilon_X(P) = \mathbb{E}_{P}(id_X)$. Henceforth we use the notation $\mathbb{E}_{\bullet}(id_X)$ for the unique affine measurable function $\G(X) \rightarrow X$ such that, for each fixed $P \in G(X)$,  the property
$$
f \big( \mathbb{E}_P(id_X)\big) = \mathbb{E}_P(f) \quad \quad \forall f \in \mathbb{R}_{\infty}^X
$$
holds. In other words, by definition,  $\mathbb{E}_P(id_X)$ is the unique point in $X$ such that the preceding equation holds.

\begin{lemma}
If $X \in_{ob} \mathbf{Cvx}_{Meas}$ then the function $\G(X) \xrightarrow{\mathbb{E}_{\bullet}(id_X)} X$ is a $\G$-algebra.
\end{lemma}
\begin{proof} 
We need to show the following two properties: 
\begin{enumerate}
\item  for all $x \in X$ we have $\mathbb{E}_{\delta_x}(id_X) = x$, and 
\item $ \mathbb{E}_{\bullet}(id_X) \circ \mu_X = \mathbb{E}_{\bullet}(id_X) \circ \G(\mathbb{E}_{\bullet}(id_X))$.
\end{enumerate}
The first property follows from  Corollary 1 and remark concerning notation.  

 For property (2) to be valid the square in the $\mathbf{Cvx}_{Meas}$-diagram
\begin{center}
\begin{tikzpicture}{scale=5}
   \node  (G2X) at (-1,0)   {$\G^2(X)$};
   \node  (GX)  at  (2,0)   {$\G(X)$};
   \node  (GX2)  at  (-1, -2)  {$\G(X)$};
  \node  (X)    at   (2, -2)  {$X$};
  \node  (R)   at   (3, -3)   {$\Rinf$};

  \draw[->,above] (G2X) to node {$\mu_X$} (GX);
  \draw[->,left] (G2X) to node {$\G(\mathbb{E}_{\bullet}(id_X))$} (GX2);
  \draw[->, below] (GX2) to node {$\mathbb{E}_{\bullet}(id_X)$} (X);
  \draw[->, left] (GX) to node {$\mathbb{E}_{\bullet}(id_X)$} (X);
  \draw[->,below,out=-60, in=180,looseness=.5] (GX2) to node {$\mathbb{E}_{\bullet}(f)$} (R);
  \draw[->,right,out=-60, in=90,looseness=.5] (GX) to node {$\mathbb{E}_{\bullet}(f)$} (R);
  \draw[->,above] (X) to node {$f$} (R);
\end{tikzpicture}
\end{center}
must commute.  Let $X \xrightarrow{f} \Rinf$ be any affine measurable map.  Use the defining property  $f \circ \mathbb{E}_{\bullet}(id_X) = \mathbb{E}_{\bullet}(f)$, and consider the commutativity of the two outer paths from $\G^2(X)$ to $\Rinf$.  Let $Q \in \G^2(X)$.   The outer east-south path is the quantity

\be 
\begin{array}{lcl}
\mathbb{E}_{\bullet}(f) \big( \mu_X(Q) \big) &=& \mathbb{E}_{\mu_X(Q)}(f) \\
& =& \int_X f \, d\mu_X(Q) \\
&=& \int_{P \in \G(X)} \big( \int_X f \, dP\big)dQ \\
&=& \int_{P \in \G(X)} \mathbb{E}_P(f)dQ
\end{array}
\ee
On the other hand, the outer south-east path  of the $\mathbf{Cvx}_{Meas}$-diagram  yields
\be
\begin{array}{lcl}
 \mathbb{E}_{\bullet}(f)\big( G(\mathbb{E}_{\bullet}(id_X))Q \big) &=&
\mathbb{E}_{G( \mathbb{E}_{\bullet}(id_X))Q}(f)  \\
&=& \mathbb{E}_Q(f \circ \mathbb{E}_{\bullet}(id_X)) \\
&=& \mathbb{E}_Q(\mathbb{E}_{\bullet}(f)) \\
&=& \int_{P \in \G(X)} \mathbb{E}_P(f) \, dQ
\end{array}
\ee
which coincides with the outer east-south path.
 The result of the lemma now follows from the property that $\mathbb{R}_{\infty}$ coseparates, i.e, the set of affine measurable functions $X  \xrightarrow{f} \mathbb{R}_{\infty}$ are jointly monic on $X$, and hence the square commutes.

\end{proof}

\begin{lemma} \label{Lemma5} Let $X \in_{ob} \mathbf{Cvx}_{Meas}$. Every affine measurable function $X \xrightarrow{f} \mathbb{R}_{\infty}$ yields a morphism of $\G$-algebras.
\end{lemma}
\begin{proof} 
We have already noted, for every $P \in \G(X)$, that   $\mathbb{E}_P(id_X) \in X$ is the unique point in $X$ such that $f(\mathbb{E}_{P}(id_X)) = \mathbb{E}_P(f)$ for all affine measurable functions $X \xrightarrow{f} \mathbb{R}_{\infty}$.
That equation is equivalent to the statement 
$$
f \circ \mathbb{E}_{\bullet}(id_X) = \mathbb{E}_{\bullet}(id_{\mathbb{R}_{\infty}}) \circ \G(f).
$$
But since both $X$ and $\mathbb{R}_{\infty}$ are objects in $\mathbf{Cvx}_{Meas}$ it follows by  Lemma 3.6 that both $\mathbb{E}_{\bullet}(id_X)$ and $\mathbb{E}_{\bullet}(id_{\mathbb{R}_{\infty}})$ are $\G$-algebras. Hence $f$ is a morphism of those algebras.
\end{proof}

\begin{lemma}
For every measurable space $X$ the space  $\G(X)$ is an object in $\mathbf{Cvx}_{Meas}$.  For every measurable function $X \xrightarrow{f} Y$ the push forward map $\G(f)$ is an affine measurable function.
\end{lemma}
\begin{proof} By Lemma 3.2 and the paragraph preceding that lemma we know that $\G(X)$ is a measurable convex space. We need to show it also satisfies the fullness property.
 
Let  $\mathbf{Cvx}_{Meas} \xrightarrow{\mathcal{U}} \mathbf{Meas}$ denote the forgetful functor which forgets the convex space structure.  Then, for every measurable set $U$ in $X$, we have the affine measurable function $ev_U$, and the property that\footnote{Note that $\mathcal{U}\big( \mathbb{E}_{\bullet}(ev_U) \big) = \hat{ \mathbb{E}}_{\bullet}(ev_U)$ since the argument $ev_U$ is an affine measurable function.   The defining property of $\mathbb{E}_{\bullet}(id_{\G(X)})$ given by $f \circ \mathbb{E}_{\bullet}(id_{\G(X)}) =\mathbb{E}_{\bullet}(f)$ for all affine measurable functions is necessary to obtain the fourth equality.   }
\be
\begin{array}{lcl}
(\mathcal{U}(ev_U) \circ \mu_X)(Q) = \mu_X(Q)[U] &=& \int_{P \in \G(X)} ev_U(P) \, dQ \\
& =& \hat{\mathbb{E}}_Q(ev_U) \\
&=& \mathcal{U}\big(\mathbb{E}_{\bullet}(ev_{\tiny{U}})\big)(Q) \\
&=&  \mathcal{U} \big(ev_U \circ \mathbb{E}_{\bullet}(id_{\G(X)})\big)(Q)
\end{array}
\ee
Since these evaluation maps are jointly monic on $\G(X)$ it follows that $\mu_X =  \mathcal{U}\big(\mathbb{E}_{\bullet}(id_{\G(X)})\big)$.   
Using this result, along with Lemma \ref{Lemma5}, we obtain the result that for $Q \in \G^2(X)$ and all affine maps $\G(X) \xrightarrow{f} \mathbb{R}_{\infty}$ the element \mbox{$\mu_X(Q) \in \cap_{f \in \mathbb{R}_{\infty}^{\G(X)}} f^{-1}(\mathbb{E}_Q(f))$} which shows the fullness property is satisfied for $\G(X)$. 
Hence  $\G(X) \in_{ob} \mathbf{Cvx}_{Meas}$.

The fact that $\G(f)$ is an affine function follows immediately by the pointwise convex space structure of $\G(X)$ and $\G(Y)$.  The fact that it is a measurable function follows because $\G$ is an endofunctor on $\mathbf{Meas}$. 
\end{proof}

By the preceding lemma we obtain  the ''free functor'' $\mathbf{Meas} \xrightarrow{\mathcal{F}} \mathbf{Cvx}_{Meas}$, which is the Giry monad (functor) $\G$ viewed as a functor into $\mathbf{Cvx}_{Meas}$. There is also a ''forgetful functor'' $\mathbf{Cvx}_{Meas}  \xrightarrow{\mathcal{U}} \mathbf{Meas}$ which forgets the convex space structure.  

\begin{lemma} Let $X \in_{ob} \mathbf{Cvx}_{Meas}$.  The affine measurable functions $\mathbb{E}_{\bullet}(id_X)$ are the components of a  natural transformation
\be \nonumber
\mathcal{F} \circ \mathcal{U} \Rightarrow id_{\mathbf{Cvx}_{Meas}}.
\ee
\end{lemma}
\begin{proof}
Suppose $X$ and $Y$ are objects in $\mathbf{Cvx}_{Meas}$ and that $X \xrightarrow{f} Y$ is an affine measurable function.  We need to verify that $f \circ \mathbb{E}_{\bullet}(id_X) = \mathbb{E}_{\bullet}(id_Y) \circ \mathcal{F}\mathcal{U}(f)$.  Applying both sides of this equation to any $P \in \mathcal{F}\mathcal{U}(X)$ yields the desired equality
$$
 f\big(\mathbb{E}_{P}(id_X)\big) = \mathbb{E}_{P}(f) = \mathbb{E}_P(id_Y \circ f)=\mathbb{E}_{\mathcal{F}\mathcal{U}(f)P}(id_Y)=\mathbb{E}_{\bullet}(id_Y)\big(\mathcal{F}\mathcal{U} (f)\big)P.
$$
Hence the required commutativity condition holds showing $\mathbb{E}$ is a natural transformation. 

\end{proof}
We denote this natural transformation by $\mathbb{E}$ rather than $\mathbb{E}_{\cdot}(id_{\bullet})$.

\begin{theorem}  $\langle \mathcal{F}, \mathcal{U}, \eta, \mathbb{E} \rangle$ forms an adjunction, and  the  monad it defines on $\mathbf{Meas}$, given by $\langle \mathcal{U} \circ \mathcal{F}, \eta, \mathcal{U}\mathbb{E}_{\mathcal{F}} \rangle$,   is the Giry monad  $(\G, \eta, \mu)$.
\end{theorem}
\begin{proof}  To show  $\langle \mathcal{F}, \mathcal{U}, \eta, \mathbb{E} \rangle$ is an adjunction
we verify the two triangular inequalities.  Let $X \in_{ob} \mathbf{Meas}$.   We need to prove that the triangle on the left-hand side of the $\mathbf{Cvx}_{Meas}$-diagram
\begin{center}
\begin{tikzpicture}{scale=5}
   \node  (FX)  at  (-1,0)  {$\mathcal{F}(X)$};
   \node  (FX2) at  (2,-2)  {$\mathcal{F}(X)$};
  \node  (FUFX) at  (2,0)  {$\mathcal{F}\mathcal{U}(\mathcal{F}(X)$};
  \node  (FUR)  at  (5.5,0)   {$\mathcal{F}\mathcal{U}(\Rinf)$};
  \node  (R)    at    (5.5, -2)  {$\Rinf$};

   \draw[->,above] (FX) to node {$\mathcal{F}(\eta_X)$} (FUFX);
  \draw[->,below] (FX) to node [xshift=-2pt]{$id_{\mathcal{F}(X)}$} (FX2);
  \draw[->,right] (FUFX) to node {$\mathbb{E}_{\bullet}(id_{\mathcal{F}X})$}  (FX2);
  \draw[->,above,dashed] (FUFX) to node {$\mathcal{F}\mathcal{U}(ev_U)$} (FUR);
  \draw[->,right,dashed] (FUR) to node {$\mathbb{E}_{\bullet}(id_{\Rinf})$} (R);
  \draw[->,below,dashed] (FX2) to node {$ev_U$} (R);
\end{tikzpicture}
\end{center}
commutes, i.e., $id_{\mathcal{F}(X)} = \mathbb{E}_{\bullet}(id_{\mathcal{F}(X)}) \circ \mathcal{F}(\eta_X)$.
Composition of those two parallel arrows by the affine measurable map $ev_U$, and using Lemma \ref{Lemma5} gives the commutative square on the right-hand side of the diagram.    Hence we seek to verify that the east-south path coincides with the diagonal-east path of the $\mathbf{Cvx}_{Meas}$-diagram.  Chasing elements we have, using $\delta_x(U)=\chi_U(x)$, 
\begin{center}
\begin{tikzpicture}{scale=5}
   \node  (FX)  at  (-1,0)  {$P$};
   \node  (FX2) at  (2,-2)  {$P$};
  \node  (FUFX) at  (2,0)  {$\mathcal{F}(\eta_X)P$};
  \node  (FUR)  at  (5.5,0)   {$\mathcal{F}\big(\mathcal{U}(ev_U) \circ \eta_X\big)P$};
  \node  (R)    at    (3.75, -2)  {$P(U)$};
  \node  (R2)    at    (5.5, -2)  {$=\int_X \delta_x(U) \, dP$};

   \draw[|->,above] (FX) to node {$$} (FUFX);
  \draw[|->,below] (FX) to node {} (FX2);
%  \draw[|->,right] (FUFX) to node {}  (FX2);
  \draw[|->,above,dashed] (FUFX) to node {} (FUR);
  \draw[|->,right,dashed] (FUR) to node {} (R2);
  \draw[|->,below,dashed] (FX2) to node {} (R);
\end{tikzpicture}
\end{center}
so $ev_U$ composed with both $id_{\mathcal{F}X}$ and $\mathbb{E}_{\bullet}(id_{\mathcal{F}(X)}) \circ \mathcal{F}(\eta_X)$ yield the same result.  This is true for every measurable set $U$ in $X$, and since the evaluation maps are jointly monic on $\mathcal{F}(X)$ it follows that the left triangle commutes.  Note that in this proof we have used the definition of a push forward probability measure to $\G(\Rinf)$, 
$$
\mathbb{E}_{\mathcal{F}\big( \mathcal{U}(ev_{\small{U}}) \circ \eta_X\big)P}(id_{\Rinf}) = \hat{\mathbb{E}}_P\big( \mathcal{U}(ev_{\tiny{U}}) \circ \eta_X\big).
$$
This is necessary to compute push forward probability measures by non-affine measurable functions $X \rightarrow \Rinf$.

To prove the second triangular equality let $A \in_{ob} \mathbf{Cvx}_{Meas}$.  We need to verify that the $\mathbf{Meas}$-diagram 
\begin{center}
\begin{tikzpicture}
  \node  (UA) at  (0,0)  {$\mathcal{U}(A)$};
  \node  (UFUA) at  (3,0)  {$\mathcal{U}\mathcal{F}(\mathcal{U}(A))$};
  \node  (UA2)   at  (3, -2)  {$\mathcal{U}(A)$};
  \draw[->,above] (UA) to node {$\eta_{\mathcal{U}(A)}$} (UFUA);
  \draw[->,right] (UFUA) to node {$\mathcal{U}(\mathbb{E}_{\bullet}(id_A))$} (UA2);
  \draw[->,below,left] (UA) to node {$id_{\mathcal{U}(A)}$} (UA2);
\end{tikzpicture}
\end{center}
commutes.  Letting $a \in \mathcal{U}(A)$ we have
\begin{center}
\begin{tikzpicture}
  \node  (UA) at  (0,0)  {$a$};
  \node  (UFUA) at  (3,0)  {$\delta_a$};
  \node  (UA2)   at  (3, -2)  {$a=\mathbb{E}_{\delta_a}(id_{\mathcal{U}(A)})$};
  \draw[|->,above] (UA) to node {} (UFUA);
  \draw[|->,right] (UFUA) to node {} (UA2);
  \draw[|->,below,left] (UA) to node {} (UA2);
\end{tikzpicture}
\end{center}
which proves the second triangular equality, and we conclude $\langle \mathcal{F}, \mathcal{U}, \eta, \mathbb{E} \rangle$ forms an adjunction.

We now proceed to show the adjunction yields the Giry monad.  It is clear that $\G = \mathcal{U} \circ \mathcal{F}$, and we have essentially already shown in the proof of Lemma 6 that $\mu_X = \mathcal{U}\big(\mathbb{E}_{\bullet}(id_{\mathcal{F}(X)})$.  We need only to alter equation (14) by replacing the term $\mathbb{E}_{\bullet}(id_{\G(X)})$ with the term $\mathbb{E}_{\bullet}(id_{\F(X)})$ which is valid because the functor $\F$ is the functor $\G$ viewed as a functor into $\mathbf{Cvx}_{Meas}$.
\end{proof}

We now apply the Beck's monadicity theorem   to prove that the comparison functor
\be \nonumber
\Phi : \mathbf{Cvx}_{Meas} \rightarrow \mathbf{Meas}^{\G}
\ee
is an equivalence. 

  Theorem 1, page 147 of \cite{CWM} is
\begin{theorem} (Beck's theorem characterizing algebras.)   Let
$$
\langle \F, \U, \eta, \epsilon \rangle: \mathcal{X} \rightarrow \mathcal{A}
$$
be an adjunction, $\langle T, \eta, \mu \rangle$ the monad which it defines in $\mathcal{X}$, 
$\mathcal{X}^T$ the category of $T$-algebras for this monad, and
$$
\langle \F^T, \U^T, \eta^T, \epsilon^T \rangle: \mathcal{X} \rightarrow \mathcal{X}^T
$$
the corresponding adjunction.  Then the following are equivalent:
\begin{enumerate}
\item The (unique) comparison functor $K: \mathcal{A} \rightarrow \mathcal{X}^T$ is an equivalence;
\item The functor $\mathcal{U}: \mathcal{A} \rightarrow \mathcal{X}$ creates coequalizers for those parallel pairs $f,g$ in $\mathcal{A}$ for which $\mathcal{U}f, \mathcal{U}g$ has an absolute coequalizer in $\X$:
\item The functor $\mathcal{U}: \A \rightarrow \X$ creates coequalizers for those parallel pairs $f,g$ in $\A$ for which $\U{f}, \U{g}$ has a split coequalizer in $\X$.
\end{enumerate}
\end{theorem}

  In Theorem 1, Chapter VI, page 152 of \cite{CWM}, MacLane provides the proof using absolute coequalizers (condition (2))  that the comparison functor $K:\operatorname{ \langle \Omega, E \rangle-\mathbf{Alg}} \rightarrow \mathbf{Set}^{\mathcal{U}'\circ \mathcal{F}'}$ between any algebraic variety and the corresponding category of algebras induced by the free functor $\mathcal{F}'$ and forgetful functor $\mathcal{U}'$ of that algebraic variety is an equivalence.  Indeed, the category $\mathbf{Cvx}$ is one such algebra, and it is convenient to view the proof as a proof that the comparison functor $K: \mathbf{Cvx} \rightarrow \mathbf{Set}^{\mathcal{U}'\circ \mathcal{F}'}$ is an equivalence. That free functor $\mathcal{F}'$ sends a set $X$ to the set of all formal finite affine sums, $\{ \sum_{i=1}^n p_i x_i \, | \, p_i \in [0,1], \sum_{i=1}^n p_i =1, x_i \in X, n \in \mathbb{N} \}$ with the natural convex space structure. By changing the base space from $\mathbf{Set}$ to $\mathbf{Meas}$, and using the fact that all the affine sum operations are measurable in $\mathbf{Cvx}_{Meas}$ rather than just set functions, that proof carries through verbatim to show that the comparison functor $\Phi:\mathbf{Cvx}_{Meas} \rightarrow \mathbf{Meas}^{\mathcal{U} \circ \mathcal{F}}$ is an equivalence.

The proof using (3) - creates coequalzers for parallel pairs - uses the following argument.
Let $A, B \in_{ob} \mathbf{Cvx}_{Meas}$,  $f,g \in_{ar} \mathbf{Cvx}_{Meas}$, and \mbox{$X \in_{ob} \mathbf{Meas}$}. We have the $\mathbf{Meas}$-diagram
\begin{center}
\begin{tikzpicture}
  \node  (UA) at (-1,0)   {$\mathcal{U}(A)$};
  \node  (UB)  at   (2,0)  {$\mathcal{U}(B)$};
  \node  (X)    at   (6,0)  {$X$};
  \node  (XX) at   (6, 2.5)  {$X \times X$};
  \node  (BB) at  (2, 2.5)  {$\mathcal{U}(B \times B)$};
   \draw[->, right,dashed] (XX) to node {$p \pi_1^{'} + (1-p) \pi_2^{'}$} (X);
  \draw[->, above] ([yshift=2pt] UA.east)  to node {$\mathcal{U}(f)$} ([yshift=2pt] UB.west);
  \draw[->, below] ([yshift=-2pt]UA.east) to node {$\mathcal{U}(g)$} ([yshift=-2pt] UB.west);
  \draw[->, below] (UB) to node {$q$} (X);
  \draw[->,above] (XX) to node {$s \times s$} (BB);
 \draw[->,left] (BB) to node {$\small{\mathcal{U}(p \pi_1 + (1-p) \pi_2)}$} (UB);
\end{tikzpicture}
\end{center}
where $q$ is the coequalizer of $\U(f)$ and $\U(g)$, and $X \xrightarrow{s} \U(B)$ is a measurable function satisfying $q \circ s=id_X$. 
We can define the convex space structure on $X$ as the composite $q \circ \mathcal{U}(p \pi_1 + (1-p) \pi_2) \circ (s \times s)$.  This makes $X \in_{ob} \mathbf{Cvx}_{Meas}$, and with this convex space structure on $X$ the function $q$ is an affine measurable function.
 Hence condition (3) holds which implies the comparison functor is an equivalence.

 That completes the proof that  $\mathbf{Cvx}_{Meas}$ is equivalent to $\mathbf{Alg}_{\G}$.

\vspace{.3in}

\bibliographystyle{plain}

\begin{thebibliography}{99}

\bibitem{CWM}
  Saunders MacLane,
  \emph{Categories for the working mathematician},
  Springer-Verlag,
  1971.

\bibitem{SFM}
  William Lawvere and Robert Rosebrugh,
  \emph{Sets for mathematics},
  Cambridge University Press,
  2003.
\end{thebibliography}

\end{document}